\documentclass[12pt]{article}

\usepackage{url}
\usepackage{amsmath}
\usepackage{amssymb}
\usepackage{amsfonts}
\usepackage{enumitem}
\usepackage{amsthm}

\usepackage{etoolbox}
\gappto{\UrlBreaks}{\UrlOrds}

\newcommand{\R}{\mathbb{R}}
\newcommand{\Z}{\mathbb{Z}}
\theoremstyle{plain}
\newtheorem{theorem}{Theorem}

\newtheorem{proposition}[theorem]{Proposition}

\theoremstyle{definition}
\newtheorem{definition}[theorem]{Definition}
\newtheorem{example}[theorem]{Example}

\begin{document}

\title{A Widespread Error in the Use of Benford's Law to Detect Election and Other Fraud}
\author{Theodore P.\ Hill}
\date{\vspace{-5ex}}  
\maketitle

\begin{abstract}

The goal of this note is to show that a widespread claim about Benford's Law, namely, that the range of every 
Benford distribution spans at least several orders of magnitude, is false. The proof is constructive and concrete examples are presented.
\end{abstract}

\section{{Introduction}}\label{sec1}

Errors often occur in mathematics, and sometimes these errors pass peer review and are published. The standard scientific method of handling such mistakes is for the author or someone else to publish corrections. Sometimes mathematical errors even make it into the general press, and when they do they can be difficult to correct. 

On rare occasions mathematical errors make it up to a still higher level of public dissemination than the standard publishing or social media outlets. When such an error is propagated not only by viral YouTube lectures but also by established  formal fact-finding teams of international news agencies and top universities alike, it becomes what this author calls a \textit{mega-error}. Mega-errors spread instantaneously with the Internet, and can be very misleading, even harmful, and exceedingly difficult to correct.

The subject of this short note addresses one such occurrence with respect to the well-known nineteenth century statistical phenomenon called Benford's Law (BL) (cf.\ \cite{BerAH15}, \cite{BerAH17}, \cite{BerAH20}, \cite{BerAHR09}, and \cite{MilS15}).  

In  \cite{BerAH10} and \cite{BerAH11B},  it was shown that this 50-year-old  widely cited claim of Feller is false:

\begin{tabbing}
\textbf{(A)} \= If a distribution is smooth and spread out, it will be close to Benford.
\end{tabbing}
\noindent

A partial converse of statement \textbf{(A)} is the assertion 
\begin{tabbing}
\textbf{(B)} \= To follow Benford's Law, a distribution must span at least several orders of magnitude. 
\end{tabbing}

As will be shown below, statement \textbf{(B)} is also false.  However, this mistake has recently been published and re-published by a number of highly-respected, high-level, and extremely widely read and cited sources including the following. \\ \\
\noindent
In the international mathematics context, a Wolfram Research mathematics blog claims
\begin{quote}
\noindent
``To observe the validity of Benford's law for a dataset, the scale of the data must extend over several orders of magnitude" \cite{TroZ10}.
\end{quote} 
 
 \noindent
 The academic and digital research coalition Election Integrity Partnership (a joint watchdog agency with Stanford University, the University of Washington, and others) claims that to follow Benford's Law,  
 \begin{quote}
 \noindent
``the numbers must span multiple orders of magnitude...Violations of these assumptions lead to violations of the law." \cite{BakWSSW20}.
\end{quote}

\noindent
This last statement by Election Integrity Partnership was picked up and repeated verbatim by the  \textit{Reuters News Agency} Fact Check Team, 

\begin{quote}
\noindent
``for the law to hold \ldots the numbers must span multiple orders of magnitude"  \cite{ReuS20}.
 \end{quote}
 
\noindent
That same day, a viral (over 1.7 million views) YouTube lecture by former maths teacher Matt Parker claimed
\begin{quote}
\noindent
``you only get Benford's Law in some situations and one of the main requirements is that the data spans several orders of magnitude" \cite{ParM20}.
 \end{quote} 
 
Thus, in this author's opinion, \textbf{(B)} can be characterized as a mega-error.  The purpose  of this note is to prove 
a general proposition that shows that claim \textbf{(B)} and the essentially equivalent news and social media claims above are false. To be Benford, a distribution or dataset only needs to span one order of magnitude; in fact a much stronger conclusion will be proved.

\section{{\bf Basic notation and definitions}}\label{sec2}

Throughout this note, the emphasis will be on decimal representations of numbers, the classical setting of Benford's law, so here $\log t$ means $\log_{10} t$, etc.  (For other bases such as binary or hexadecimal, the analogous results hold with very little change, simply replacing $\log = \log_{10}$ by $\log_b$ for the appropriate base $b$.  The interested reader is referred to \cite[p.~9]{BerAH15} for details.)

The basic notion underlying Benford's law concerns the {\em leading significant digits}  and, more generally, the {\em significand} of  a number (also sometimes referred to in scientific notation as the \textit{mantissa}). 

\begin{definition}\label{def21} 
For $x \in \R^+$, the (decimal) {\em significand of x}, $S(x)$, is given by $S(x) = t$, where $t$ is the unique number in $[1, 10)$ with $x = 10^k t$ for some (necessarily unique) $k \in \Z$. For negative $x$, $S(x) = S(-x)$, and for convenience, $S(0)=0$.
\end{definition}

\begin{example}\label{ex22}
$S(2019) = 2.019 = S(0.02019) =S(-20.19) $.
\end{example}

\begin{definition}\label{def23}
{\em The first (decimal) significant digit of a number} $x$, denoted by $D{_1}(x)$, is the first (left-most) digit of $S(x)$.  Similarly, $D{_2}(x)$ denotes the second digit of $S(x)$,  $D_3(x)$ the third digit, etc.
\end{definition}

\begin{example}\label{ex24}
$D{_1}(2019) = D{_1}(0.0219) = D{_1}(-20.19) = 2, D{_2}(2019) = 0, D{_3}(2019) = 1, D{_4}(2019) = 9$, and $D{_k}(2019) = 0$ for all $k > 4$.
\end{example}

\begin{definition}\label{def25}
A random variable $X$ is {\em Benford} if 
$$
P(S(|X|) \leq t) = \log t \quad \mbox{\rm for all } t \in [1, 10).
$$
\end{definition}
For real numbers $a < b$, let $U[a,b]$ denote a random variable that is uniformly distributed on $[a,b]$, i.e., 
$P(U[a,b] \leq x) = 0$ for $x<a$, = $(x-a)/(b-a)$ for $x \in [a,b]$, and $ = 1$ for $x>b$.

\begin{example}
\label{ex25}
\cite[Example 3.6(i)]{BerAH15} $X=10^{U[0,1]}$ is Benford. 
\end{example}
 
\begin{example}
In the special case of first significant digits, Definition \ref{def25} yields  the well-known {\em first-digit law}: For every Benford random variable $X$, 

$$
P(D_1(X) = d) = 
log(1 + \frac{1}{d})
 \quad \mbox{\rm for all } d \in \{1,2, \ldots, 9\}.
$$
\end{example}
\noindent
N.B.  None of the classical probability distributions are exactly Benford, although some are close for certain values of their parameters.  For example, no uniform, exponential, normal, or Pareto random variables are exactly Benford, although Pareto and log normal random variables, among others, can be arbitrarily close to Benford depending on the values of their parameters.

\section{Main Proposition}
The following proposition, the main result in this note, shows that both \textbf{(B)} and its converse are false, in a very strong sense. 

\begin{proposition} \label{prop3.1}
Fix $b>a>0$. 

\begin{enumerate}[label=(\roman*)]
\item If $b<10a$, no random variable with range contained in $[a,b]$ is Benford;
\item If $b=10a$, there is exactly one Benford distribution with range in $[a,b]$;
\item If $b>10a>0$, then for each $c \in (0,0.1b- a))$,
$X_c = (a+c)10^{U[0,1]}$ has range contained in $[a,b]$ and is  Benford; and
 for each $c \in (0, b-10a)$, $Y_c = U[a+c, 10a+c]$ has range contained in $[a,b]$ and is not Benford.
\end{enumerate}
\end{proposition}

\begin{proof}

For (i), let $X$ be a distribution with range in $[a,b]$, where $b<10a$.  Then it is easy to see that there are certain sequences of significant digits that the range of $X$ does not include, which violates Definition \ref{def25}.  

Conclusion (ii) follows easily since a random variable $X$ is Benford if and only if its logarithm is uniformly distributed mod 1 \cite[Theorem 4.2]{BerAH15}, which implies that if $X$ has range in $[a, 10a]$, it necessarily has the distribution $X=a10^{U[0,1]}$. 

To see (iii)  first note that by Example \ref{ex25}, $10^{U[0,1]}$ is Benford.  Since Benford's Law is scale-invariant \cite[Theorem 5.3]{BerAH15} this implies that $(a+c)10^{U[0,1]}$ is a Benford random variable.  For the second conclusion in (iii), note that by \cite[Proposition 1]{BerAH11B}, all uniform distributions are bounded strictly away from Benford's Law; in particular no uniformly distributed random variable is Benford.  
\end{proof}

Thus for every range of values spanning one order of magnitude, there is a Benford random variable with that range, and if the range is more than one order of magnitude, even only infinitesimally more, there will be infinitely many Benford and also infinitely many non-Benford distributions with that range.  

\begin{example}
\label{ex3.1}
\begin{enumerate}[label=(\roman*)]
\item There are no Benford random variables (datasets, distributions) with values only between 100 and 999, or between 73 and 729.99. 
\item There is exactly one Benford distribution, namely $X=100(10^{U[0,1]})$ with range in [100, 1000], and only one,
$X=73(10^{U[0,1]})$, with range in [73, 730].
\item There are infinitely many Benford distributions and infinitely many non-Benford distributions with values in $[100, 1000.0001]$ and also in the range $[73, 730.01]$. 
\end{enumerate} 
\end{example}

The Benford random variables $X_c$ in Proposition \ref{prop3.1}(iii) all have regular and continuous densities, but there are many Benford distributions which do not have this regular form; the next example exhibits one such distribution.
 
\begin{example}
\label{pupExample}
\cite[Example 3.6 (iii)]{BerAH15} The absolutely continuous random variable $X$ with density
$$
f_X(x) = \left\{ \begin{array}{ll}
x^{-2} (x-1)\log e  & \mbox{\rm if } x \in [1,10) \, ,\\[1mm]
10x^{-2}\log e   & \mbox{\rm if } x \in [10,100) \, ,\\[1mm]
0 & \mbox{otherwise} \, ,
\end{array} \right.
$$
is Benford, since the random variable $S(X)$ is also absolutely continuous, with density
$t^{-1}\log e$ for $t \in [1, 10)$,  even though $X$ is not of the
type of distribution considered in Proposition \ref{prop3.1}(i) above.
\end{example}

\vspace{1em}
\noindent
Note:  For a comprehensive treatment of BL at the level of advanced mathematics (measure theory, complex variables, functional analysis, etc.) the reader is referred to \cite{BerAH15}, and for a recent overview of BL at the level of beginning calculus and probability (without proofs), see \cite{BerAH20}.  
 

\end{document}